\documentclass[a4paper,leqno]{amsart}

\usepackage{amssymb,amsthm,amsmath,amsrefs}

\usepackage{ifthen}

\newcounter{mylisti} \newcounter{mylistii}
\newcounter{nest}
\newcommand{\defaultlabel}{}

\newenvironment{mylist}[1]{%
  \addtocounter{nest}{1}
  \ifthenelse{\value{nest}=1}{%
    \renewcommand{\defaultlabel}{(\roman{mylisti})\hfill}}{%
    \renewcommand{\defaultlabel}{(\alph{mylistii})\hfill}}
  \begin{list}{\defaultlabel}{%
      \ifthenelse{\value{nest}=1}{\usecounter{mylisti}}{%
        \usecounter{mylistii}}
      
      \addtolength{\itemsep}{0.5ex}
      \settowidth{\labelwidth}{#1}
      \setlength{\leftmargin}{\labelwidth}
      \addtolength{\leftmargin}{\labelsep}}}{\addtocounter{nest}{-1}
\end{list}}

\newcommand{\bn}{\ensuremath{\mathbb N}}

\newcommand{\br}{\ensuremath{\mathbb R}}

\newcommand{\cA}{\ensuremath{\mathcal A}}

\newcommand{\cF}{\ensuremath{\mathcal F}}
\newcommand{\cG}{\ensuremath{\mathcal G}}

\newcommand{\cS}{\ensuremath{\mathcal S}}

\newcommand{\cU}{\ensuremath{\mathcal U}}

\newcommand{\At}{\ensuremath{\tilde{A}}}
\newcommand{\Bt}{\ensuremath{\tilde{B}}}

\newcommand{\abs}[1]{\lvert #1\rvert}

\newcommand{\bi}{\ensuremath{\boldsymbol{1}}}

\newcommand{\intp}[1]{\ensuremath{\lfloor #1\rfloor}}

\newcommand{\ip}[2]{\ensuremath{\langle #1,#2\rangle}}
\newcommand{\bigip}[2]{\ensuremath{\big\langle #1,#2\big\rangle}}
\newcommand{\Bigip}[2]{\ensuremath{\Big\langle #1,#2\Big\rangle}}

\newcommand{\norm}[1]{\lVert #1\rVert}
\newcommand{\bignorm}[1]{\big\lVert #1\big\rVert}
\newcommand{\Bignorm}[1]{\Big\lVert #1\Big\rVert}

\newcommand{\tnorm}[1]{\lvert\mspace{-1mu}\lvert\mspace{-1mu}\lvert
  #1\rvert\mspace{-1mu}\rvert\mspace{-1mu}\rvert}

\newcommand{\cspn}{\ensuremath{\overline{\mathrm{span}}}}

\newcommand{\supp}{\operatorname{supp}}

\newcommand{\co}{\mathrm{c}_0}
\newcommand{\coo}{\mathrm{c}_{00}}

\newcommand{\vare}{\varepsilon}
\newcommand{\varf}{\varphi}

\newcommand{\ds}{\displaystyle}

\newcommand{\ts}{\textstyle}
\newcommand{\phtm}[1]{\text{\makebox[0pt]{\phantom{$#1$}}}}

\newcommand{\ie}{\textit{i.e.,}\ }

\newcommand{\cf}{\textit{cf.}\ }

\newtheorem{thm}{Theorem}
\newtheorem{mainthm}{Theorem}

\newtheorem{mainproblem}[mainthm]{Problem}

\newtheorem{lem}[thm]{Lemma}
\newtheorem{prop}[thm]{Proposition}
\newtheorem{cor}[thm]{Corollary}
\newtheorem{problem}[thm]{Problem}

\theoremstyle{definition}

\theoremstyle{remark}

\newtheorem*{rem}{Remark}

\begin{document}

\allowdisplaybreaks

\title{Renorming spaces with greedy bases}

\author{S.~J.~Dilworth}
\address{Department of Mathematics, University of South Carolina,
  Columbia, SC 29208, USA}
\email{dilworth@math.sc.edu}

\author{D.~Kutzarova}
\address{Institute of Mathematics, Bulgarian Academy of
Sciences, Sofia, Bulgaria.}
\curraddr{Department of Mathematics, University of Illinois at   
Urbana-Champaign, Urbana, IL 61801, USA.}
\email{denka@math.uiuc.edu}

\author{E.~Odell}

\author{Th.~Schlumprecht}
\address{Department of Mathematics, Texas A\&M University, College
  Station, TX 77843, USA and Faculty of Electrical Engineering, Czech
  Technical University in Prague,  Zikova 4, 166 27, Prague}
\email{thomas.schlumprecht@math.tamu.edu}

\author{A.~Zs\'ak}
\address{Peterhouse, Cambridge, CB2 1RD, UK}
\email{a.zsak@dpmms.cam.ac.uk}

\date{15 March 2014}

\thanks{Edward Odell (1947-2013). The author passed away during the
  production of this paper.}
\thanks{The first author's research was supported by NSF grant
  DMS1101490. The fourth author's research was supported by NSF grant
  DMS1160633. The first, second and fifth authors were supported by
  the Workshop in Analysis and Probability at Texas A\&M University in
  2013. The fifth author was supported by Texas A\&M University while
  he was Visiting Scholar there in 2014.} 
\keywords{Greedy bases, democratic bases, renorming}
\subjclass[2010]{41A65, 41A44, 41A50, 46B03}

\begin{abstract}
  We study the problem of improving the greedy constant or the
  democracy constant of a basis of a Banach space by renorming. We
  prove that every Banach space with a greedy basis can be renormed,
  for a given $\vare>0$, so that the basis becomes
  $(1+\vare)$-democratic, and hence $(2+\vare)$-greedy, with respect
  to the new norm. If in addition
  the basis is bidemocratic, then there is a renorming so that in the
  new norm the basis is $(1+\vare)$-greedy. We also prove that in the
  latter result the additional assumption of the basis being
  bidemocratic can be removed for a large class of bases. Applications
  include the Haar systems in $L_p[0,1]$, $1<p<\infty$, and in dyadic
  Hardy space $H_1$, as well as the unit vector basis of Tsirelson
  space.
\end{abstract}

\maketitle

\section{Introduction}

In approximation theory one is often faced with the following
problem. We start with a signal, \ie a vector $x$ in some
Banach space $X$. We then consider the (unique) expansion
$\sum_{i=1}^\infty x_i e_i$ of $x$ with respect to some (Schauder)
basis $(e_i)$ of $X$. For example, this may be a Fourier expansion of
$x$, or it may be a wavelet expansion in $L_p$. We then wish to
approximate $x$ by considering $m$-term approximations with respect to
the basis. The smallest error is given by
\[
\sigma_m(x) = \inf \bigg\{ \Bignorm{x-\sum_{i\in A}
  a_ie_i}:\,A\subset\bn,\ \abs{A}\leq m,\ (a_i)_{i\in A}\subset \br
\bigg\}\ .
\]
We are interested in algorithms that are easy to implement and that
produce the best $m$-term approximation, or at least get close to
it. A very natural process is the greedy algorithm which we now
describe. For each $x=\sum x_ie_i\in X$ we fix a permutation
$\rho=\rho_x$ of $\bn$ (not necessarily unique) such that
$\abs{x_{\rho(1)}}\geq \abs{x_{\rho(2)}}\geq \dots$. We then define
the \emph{$m^{\text{th}}$ greedy approximant to $x$ }by
\[
\cG_m(x)=\sum_{i=1}^m x_{\rho(i)} e_{\rho(i)}\ .
\]
For this to make sense we need $\inf \norm{e_i}>0$, otherwise $(x_i)$
may be unbounded. In fact, since we will be dealing with democratic
bases, all our bases will be \emph{seminormalized}, which means that
$0<\inf \norm{e_i}\leq \sup\norm{e_i} <\infty$. It follows that the
biorthogonal functionals $(e^*_i)$ are also seminormalized. Note that
a space with a seminormalized basis $(e_i)$ can be easily renormed to
make $(e_i)$ \emph{normalized}, \ie $\norm{e_i}=1$ for all $i\in\bn$.

We measure the efficiency of the greedy algorithm by comparing it to
the best $m$-term approximation. We say that $(e_i)$ is a \emph{greedy
  basis }for $X$ if there exists $C>0$ (\emph{$C$-greedy}) such that
\[
\norm{x-\cG_m(x)} \leq C\sigma_m(x)\qquad \text{for all }x\in X\text{
  and for all }m\in\bn\ .
\]
The smallest $C$ is the \emph{greedy constant }of the basis. Note that
being a greedy basis is a strong property. It implies in particular
the strictly weaker property that $\cG_m(x)$ converges to $x$ for all
$x\in X$. If this weaker property holds, then we say that the basis
$(e_i)$ is \emph{quasi-greedy}. This is still a non-trivial property:
a Schauder basis need not be quasi-greedy in general.

The simplest examples of greedy bases include the unit vector basis
of $\ell_p$ ($1\leq p<\infty$) or $\co$, or orthonormal bases of a
separable Hilbert space. An important and non-trivial example is the
Haar basis of $L_p[0,1]$ ($1<p<\infty$) which was shown to be greedy
by V.~N.~Temlyakov~\cite{temlyakov:98}. This result was later
established by P.~Wojtaszczyk~\cite{wojtaszczyk:00} using a different
method which extended to the Haar system in one-dimensional dyadic
Hardy space $H_p(\br),\ 0<p\leq 1$. We also mention two recent
results. S.~J.~Dilworth, D.~Freeman, E.~Odell and
Th.~Schlumprecht~\cite{dfos:11} proved that $\big(\oplus_{n=1}^\infty
\ell_p^n\big)_{\ell_q}$ has a greedy basis whenever $1\leq p\leq
\infty$ and $1<q<\infty$. Answering a question raised
in~\cite{dfos:11}, G.~Schechtman showed that none of the space
$\big(\bigoplus_{n=1}^\infty \ell_p \big)_{\ell_q}$, $1\leq p\neq
q<\infty$, $\big(\bigoplus_{n=1}^\infty \ell_p \big)_{\co}$, $1\leq
p<\infty$, and $\big(\bigoplus_{n=1}^\infty \co \big)_{\ell_q}$, $1\leq
q<\infty$, have greedy bases.

Greedy bases are closely related to unconditional
bases. We recall that a basis $(e_i)$ of a Banach space $X$ is said to
be \emph{unconditional }if there is a constant $K$
(\emph{$K$-unconditional}) such that
\[
\Bignorm{\sum a_i e_i} \leq K\cdot \Bignorm{\sum b_i e_i} \qquad
\text{whenever }\abs{a_i}\leq\abs{b_i} \text{ for all }i\in\bn\ .
\]
The best constant $K$ is the \emph{unconditional constant }of the
basis which we denote by~$K_U$. The property of being unconditional is
easily seen to be equivalent to that of being \emph{suppression
  unconditional }which means that for some constant $K$
(\emph{suppression $K$-unconditional}) the natural projection onto any
subsequence of the basis has norm at most $K$:
\[
\Bignorm{\sum _{i\in A} a_i e_i} \leq K \cdot
\Bignorm{\sum_{i=1}^\infty a_i e_i} \qquad\text{for all
}(a_i)\subset\br,\ A\subset\bn\ .
\]
The smallest $K$ is the \emph{suppression unconditional constant }of
the basis and is denoted by~$K_S$. It is easy to verify that $K_S\leq
K_U\leq 2K_S$. Note that it is trivial to renorm the space $X$ so that
in the new norm the basis is suppression
$1$-unconditional. Indeed, for $A\subset \bn$ the map $P_A\colon X\to
X$, defined by $P_A\Big( \sum_{i\in\bn} x_ie_i \Big)=\sum_{i\in A}
x_ie_i$, is bounded in norm by $K_S$. Hence
\[
\tnorm{x}=\sup \{ \norm{P_A(x)}:\, A\subset\bn \}
\]
is a $K_S$-equivalent norm on $X$ in which $(e_i)$ is suppression
$1$-unconditional. Similarly, using maps $M_\lambda\colon X\to
X$ given by $\sum x_ie_i\mapsto \sum\lambda_i x_i e_i$, where
$\lambda=(\lambda_i)\in B_{\ell_\infty}$, we can define a
$K_U$-equivalent norm on $X$ in which $(e_i)$ is $1$-unconditional.

In~\cite{kony-tem:99}, S.~V.~Konyagin and
V.~N.~Temlyakov introduced the notion of greedy and democratic bases
and proved the following characterization.
\begin{thm}[{\cite{kony-tem:99}*{Theorem~1}}]
  \label{thm:kt-greedy-char}
  A basis of a Banach space is greedy if and only if it is
  unconditional and democratic.
\end{thm}
A basis $(e_i)$ is said to be \emph{democratic }if there is a constant
$\Delta\geq 1$ (\emph{$\Delta$-democratic}) such that
\[
\Bignorm{\sum_{i\in A}e_i} \leq \Delta \Bignorm{\sum_{i\in
    B}e_i}\qquad \text{whenever }\abs{A}\leq\abs{B}\ .
\]
By carefully following the proof of~\cite{wojtaszczyk:03}*{Theorem~1},
one obtains the following estimates:
\[
K_S\leq C\ ,\quad \Delta\leq C\qquad\text{and}\qquad C\leq
K_S+K_S K_U^2\cdot\Delta\ .
\]
One can in fact get slightly better estimates by amalgamating some of
the steps in that proof:
\begin{equation}
  \label{eq:connections-between-constants}
  K_S\leq C\ ,\quad \Delta\leq C\qquad\text{and}\qquad C\leq
  K_S+K_U^2\cdot\Delta\ .
\end{equation}
That is, a $C$-greedy basis is suppression $C$-unconditional and
$C$-democratic, and conversely, an unconditional and
$\Delta$-democratic basis is $C$-greedy with $C\leq K_S+K_U^2\cdot
\Delta$. In particular a $1$-unconditional, $1$-democratic basis is
$2$-greedy. By~\cite{dosz:11}*{Theorem~3.1} the constant~$2$ is best
possible. Thus, improving the democracy constant by renorming will not
in general improve the greedy constant beyond~$2$.


In this paper we are concerned with the problem whether a Banach space
$X$ with a greedy basis $(e_i)$ can be renormed so that in the new
norm the greedy constant of the basis $(e_i)$ is improved ideally
to~$1$ or at least to $1+\vare$ where $\vare>0$ can be chosen
arbitrarily small. As a byproduct, we also obtain results on
renormings that improve the democracy constant. The maps $P_A$ and
$M_\lambda$ that were used above in renormings that improve the
unconditional constants are linear. By contrast, the functions $\cG_m$
that map vectors to their greedy approximants are not linear, and that
is what makes the problem of improving the greedy constant far from
trivial.

In the rest of this section we recall what is already known about this
problem and state our new results. Definitions will be given in later
sections when needed.

In~\cite{alb-woj:06} F.~Albiac and P.~Wojtaszczyk gave a characterization
of $1$-greedy bases in terms of a weak symmetry property of the
basis. They raised several open problems about symmetry properties of
$1$-greedy bases and about the possibility of improving greedy and
democratic constants by renorming. Most of the problems were answered
by four of the authors of this paper in~\cite{dosz:11}. In
Section~\ref{sec:bidemocratic} we recall the Albiac-Wojtaszczyk
characterization, and a theorem from~\cite{dosz:11} which shows that
a space with an unconditional, bidemocratic basis can be renormed to
make the basis $1$-unconditional and
$1$-bidemocratic. By~~\eqref{eq:connections-between-constants} above,
such a basis is $2$-greedy. Here we will obtain the following stronger
result.
\begin{mainthm}
  \label{mainthm:1+e-greedy-renorming-for-bidemocratic}
  Let $X$ be a Banach space with an unconditional, bidemocratic
  basis $(e_i)$. Then for all $\vare>0$ there is an equivalent norm on $X$
  with respect to which $(e_i)$ is $1$-unconditional, $1$-bidemocratic
  and $(1+\vare)$-greedy.
\end{mainthm}
In particular, the above result applies to the Haar basis of
$L_p[0,1]$ for $1<p<\infty$. In~\cite{alb-woj:06} Albiac and
Wojtaszczyk raise the problem whether $L_p[0,1]$ can be renormed so
that the Haar basis becomes $1$-greedy in the new norm. This problem
is still open. The result above gets close to giving a positive
answer.

Section~\ref{sec:general} is concerned with the general case, \ie when
we do not assume bidemocracy. In~\cite{alb-woj:06} Albiac and
Wojtaszczyk asked whether the democracy constant can be improved
to~$1$. It was already shown in~\cite{dosz:11} that the answer in
general is `no': the Haar system of dyadic $H_1$, or an arbitrary
unconditional basis of Tsirelson's space $T$ cannot be made
$1$-democratic by renorming. Here we are able to prove the following
positive result.
\begin{mainthm}
  \label{mainthm:1+e-democratic}
  Let $(e_i)$ be an unconditional and democratic basis of a Banach
  space $X$. For any $\vare>0$ there is an equivalent norm on $X$ with
  respect to which $(e_i)$ is normalized, $1$-unconditional and
  $(1+\vare)$-democratic.
\end{mainthm}
This answers a question raised by W.~B.~Johnson.
By equation~\eqref{eq:connections-between-constants}, it follows from
this theorem that if $(e_i)$ is a greedy basis of a Banach space $X$,
then for all $\vare>0$ there is an equivalent norm on $X$ with respect
to which $(e_i)$ is $1$-unconditional and $(2+\vare)$-greedy. The
following problem remains open in its full generality.
\begin{mainproblem}
  \label{problem:1+e-greedy-renorming-general}
  Let $X$ be a Banach space with a greedy basis $(e_i)$. Given
  $\vare>0$, does there exist an equivalent norm on $X$ with respect
  to which $(e_i)$ is $1$-unconditional and $(1+\vare)$-greedy?
\end{mainproblem}
In the last section we will give a positive answer for a large class
of bases. As an application we obtain, for any $\vare>0$, a renorming
of Hardy space $H_1$ and of Tsirelson's space $T$ such that the Haar
system, respectively, unit vector basis is $(1+\vare)$-greedy.

\section{Bidemocratic bases}
\label{sec:bidemocratic}

The aim of this section is to prove
Theorem~\ref{mainthm:1+e-greedy-renorming-for-bidemocratic}.
We first recall the Albiac-Wojtaszczyk characterization of $1$-greedy
bases~\cite{alb-woj:06}. In fact a trivial modification of their proof
gives a characterization of $C$-greedy bases for an arbitrary
$C\geq 1$. For the sake of completeness we shall state and prove their
result here in that more general form.

Let $(e_i)$ be a basis of a Banach space $X$ with biorthogonal
sequence $(e^*_i)$. For a finite set $A\subset\bn$ we denote by
$\bi_A$ the vector $\sum_{i\in A} e_i$ of $X$ or sometimes the vector
$\sum_{i\in A} e^*_i$ in $X^*$. It will be clear from the context
which one is meant. For example the notation $\norm{\bi_A}$ means the
norm of $\sum_{i\in A} e_i$ in $X$, whereas $\norm{\bi_A}^*$ indicates
the norm in the dual space of $\sum_{i\in A} e^*_i$. The
\emph{support} with
respect to the basis $(e_i)$ of a vector $x=\sum x_ie_i$ in $X$ 
is the set $\supp(x)=\{ i\in\bn:\,x_i\neq 0\}$.  The subspace of
vectors with finite support, \ie the linear span of $(e_i)$, can be
indentified in the obvious way with the space $\coo$ of real sequences
that are eventually zero. The basis $(e_i)$ then corresponds to the
unit vector basis of $\coo$. Given vectors $x=\sum x_ie_i$ and $y=\sum
y_ie_i$ in $\coo$, we say \emph{$y$ is a greedy rearrangement of $x$
}if there exist $w, u=(u_i), t=(t_j)\in\coo$ of pairwise disjoint
support such that $x=w+u$, $y=w+t$, $\abs{\supp(u)}=\abs{\supp(t)}$,
and $\norm{w}_{\ell_\infty}\leq \abs{u_i}=\abs{t_j}$ for all
$i\in\supp(u),\ j\in\supp(t)$. To put it informally, $y$ is obtained
from $x$ by moving (and possibly changing the sign of) some of the
coefficients of $x$ of maximum modulus to co-ordinates where $x$ is
zero. Given $C\geq 1$, we say that $(e_i)$ \emph{has Property~(A) with
  constant $C$ }if for all $x,y\in\coo$ we have $\norm{y}\leq
C\norm{x}$ whenever $y$ is a greedy rearrangement of $x$.
\begin{thm}[{\cf \cite{alb-woj:06}*{Theorem~3.4}}]
  \label{thm:greedy-char}
  Let $(e_i)$ be a basis of a Banach space $X$. If $(e_i)$ is
  $C$-greedy, then it is suppression $C$-unconditional and has
  Property~(A) with constant~$C$. Conversely, if $(e_i)$ is
  suppression $K$-unconditional and has Property~(A) with
  constant~$C$, then it is greedy with constant at most $K^2C$.

  In particular, a suppression $1$-unconditional basis of a Banach
  space is $C$-greedy if and only if it satisfies Property~(A) with
  constant~$C$.
\end{thm}
\begin{proof}
  First assume that $(e_i)$ is $C$-greedy. We show that if $x=\sum
  x_ie_i\in X$ has finite support and $A\subset \bn$, then
  $\bignorm{\sum_{i\in A} x_ie_i}\leq C\norm{x}$. Let
  $B=\supp(x)\setminus A$ and $m=\abs{B}$. Choose a real number
  $\lambda>\norm{x}_{\ell_\infty}$, and set $z=\sum_{i\in A} x_ie_i +
  \lambda\bi_B$. Then $\cG_m(z)=\lambda\bi_B$ and $w=\sum_{i\in
    B}(\lambda-x_i)e_i$ is an $m$-term approximation to $z$. It
  follows that
  \[
  \Bignorm{\sum_{i\in A} x_ie_i} = \norm{z-\cG_m(z)} \leq C\norm{z-w}
  = C\norm{x}\ ,
  \]
  as required. We next show that $(e_i)$ has Property~(A) with
  constant~$C$. Let $y=w+t$ be a greedy rearrangement of $x=w+u$,
  where $w, u, t\in\coo$ are as in the definition above. Fix
  $\delta>0$ and set $z=w+(1+\delta)u+t$. Let
  $m=\abs{\supp(u)}=\abs{\supp(t)}$. Then $\cG_m(z)=(1+\delta)u$,
  whereas $t$ is another $m$-term approximation to $z$. It follows
  that
  \[
  \norm{y}=\norm{z-\cG_m(z)} \leq C\norm{z-t} =C
  \norm{x+\delta u}\ .
  \]
  Letting $\delta\to 0$ yields $\norm{y}\leq C\norm{x}$, as required.

  To prove the converse, fix $x=\sum x_ie_i\in \coo$ and $m\in\bn$. Let
  $\sum_{i\in A} x_ie_i$ be the $m^{\text{th}}$ greedy approximant to $x$,
  and let $b=\sum_{i\in B} b_ie_i$ be an arbitrary $m$-term
  approximation. Let $s=\min \{\abs{x_i}:\,i\in A\}$, and for each
  $i\in\bn$ let $\vare_i$ be the sign of $x_i$. Note that
  $\abs{x_i}\geq s\geq \abs{x_j}$ for all $i\in A$ and $j\notin A$.
  The following is a well known consequence of suppression
  $K$-unconditionality. If $0\leq y_i\leq z_i$ or $z_i\leq y_i\leq 0$
  for all $i\in\bn$, then $\bignorm{\sum y_ie_i}\leq K\bignorm{\sum
    z_ie_i}$. We use this in the first and third inequalities below,
  whereas the second inequality uses Property~(A).
  \begin{align*}
    \norm{x-b} & = \Bignorm{\sum_{i\in A\setminus B}x_ie_i + \sum
      _{i\in B} (x_i-b_i) + \sum_{i\notin A\cup B} x_ie_i}\\[2ex]
    &\geq \frac{1}{K} \Bignorm{\sum_{i\in A\setminus B}s \vare_ie_i +
      \sum_{i\notin A\cup B} x_ie_i} \geq \frac{1}{KC}
    \Bignorm{\sum_{i\in B\setminus A}s \vare_i e_i + \sum_{i\notin
        A\cup B} x_ie_i} \\[2ex]
    & \geq \frac{1}{K^2C} \Bignorm{\sum_{i\in B\setminus A}x_i e_i +
      \sum_{i\notin A\cup B} x_ie_i} 
    =\frac{1}{K^2C}\bignorm{x-\cG_m(x)}\ .
  \end{align*}
  This completes the proof.
\end{proof}
\begin{rem}
  Let $(e_i)$ be a $1$-unconditional basis of a Banach space $X$. For
  $x\in\coo$ define $\norm{\cdot}_x$ to be the function
  \[
  \norm{z}_x=\bignorm{z+\norm{z}_{\ell_\infty}\cdot x}\ ,
  \]
  which defines a norm on $\cspn
  \{e_i:\,i\in\bn\setminus\supp(x)\}$. Theorem~\ref{thm:greedy-char}
  implies that $(e_i)$ is $C$-greedy if and only if for every
  $x\in\coo$ with $\norm{x}_{\ell_\infty}\leq 1$ the norm
  $\norm{\cdot}_x$ is $C$-democratic. This characterization of greedy
  bases is slightly different from the one given by Konyagin and
  Temlyakov~\cite{kony-tem:99} where they only assume the democracy of
  $\norm{\cdot}_x$ for $x=0$. However, for our purposes, the above
  result has the advantage that the greedy constant is the same as the
  Property~(A) constant.
\end{rem}
We next recall the notion of bidemocracy, which was introduced by
S.~J.~Dilworth, N.~J.~Kalton, Denka Kutzarova and V.~N.~Temlyakov
in~\cite{dkkt:03}, and the corresponding
renorming result~\cite{dosz:11}*{Theorem~2.1}. Suppose that $(e_i)$ is
a seminormalized basis of a Banach space $X$ with biorthogonal
sequence $(e_i^*)$. The \emph{fundamental function }$\varf$ of $(e_i)$
is defined by
\begin{equation*}
  \varf(n) = \sup_{\abs{A} \leq n}\Bignorm{\sum_{i \in A} e_i}\ .
\end{equation*}
The \emph{dual fundamental function }$\varf^*$ is given by
\begin{equation*}
  \varf^{*}(n) = \sup_{\abs{A} \leq n}\Bignorm{\sum_{i \in A}
    e_i^*}\ . 
\end{equation*}
We recall that $(\varf(n)/n)$ is a decreasing function of $n$, since
for any $A \subset \bn$ with $\abs{A} = n\geq 2$ we have
\begin{equation*}
  \Bignorm{\sum_{i \in A} e_i} =
  \frac{1}{n-1}\Bignorm{\sum_{\phtm{j\in A\setminus\{i\}} i \in A}\
    \sum_{j\in A\setminus\{i\}} e_j} \leq \frac{n}{n-1} \varf(n-1)\ .
\end{equation*}
Clearly, $\varf(n) \varf^*(n) \geq n$. We say that $(e_i)$ is
\emph{bidemocratic }if there is a constant $\Delta\geq 1$
(\emph{$\Delta$-bidemocratic}) such that
\begin{equation*}
  \varf(n) \varf^*(n) \leq \Delta n \qquad\text{for all }n\in\bn\ .
\end{equation*}
It is known~\cite{dkkt:03}*{Proposition~4.2} that if $(e_i)$ is
bidemocratic with constant $\Delta$, then both $(e_i)$ and $(e_i^*)$
are democratic with constant $\Delta$. In~\cite{dosz:11} the following
result was proved.
\begin{thm}
  \label{thm:bidemocratic-renorming}
  Suppose that $(e_i)$ is a $1$-unconditional and
  $\Delta$-bidemocratic basis for a Banach space $X$. Then
  \begin{equation}
    \label{eq:1democratic-renorming}
    \tnorm{x} = \max \Big\{ \norm{x},\ \sup_{\abs{A} < \infty}
    {\ts\frac{\varf(\abs{A})}{\abs{A}}} \sum_{i \in A}
    \abs{e_i^*(x)}\Big\}
  \end{equation}
  is an equivalent norm on $X$. Moreover, $(e_i)$ is $1$-unconditional
  and $1$-bidemocratic with respect to $\tnorm{\cdot}$. In particular,
  $(e_i)$ and $(e_i^*)$ are $1$-democratic and $2$-greedy.
\end{thm}
By~\cite{dosz:11}*{Theorem~3.1}, the conclusion that $(e_i)$ is
$2$-greedy whenever it is $1$-unconditional and $1$-democratic cannot
be strengthened in general. We now prove a stronger theorem which is the
main result of this
section. First we introduce two pieces of notation. For a vector
$x=\sum x_ie_i$ we write $\abs{x}$ for $\sum \abs{x_i} e_i$, and
$x\geq 0$ if $x_i\geq 0$ for all $i\in\bn$.
\begin{thm}
  \label{thm:1+e-greedy-renorming-for-bidemocratic}
  Let $X$ be a Banach space with an unconditional, bidemocratic
  basis $(e_i)$. Then for all $\vare>0$ there is an equivalent norm on $X$
  with respect to which $(e_i)$ is $1$-unconditional, $1$-bidemocratic
  and $(1+\vare)$-greedy.
\end{thm}
\begin{proof}
  After renorming, we may assume that $(e_i)$ is normalized,
  $1$-unconditional and $1$-bidemocratic. Let $\varf$ and $\varf^*$
  denote the fundamental
  and, respectively, dual fundamental function of $(e_i)$. Fix
  $\vare\in(0,1)$. Define a new norm $\tnorm{\cdot}$ on $X$ as
  follows.
  \[
  \tnorm{x} = \sup \Big\{
  \bigip{\abs{x}}{x^*+{\ts\frac{1}{\varf^*(n)}}\bi_A}:\,x^*\in \vare
  B_{X^*},\ n\in\bn,\ A\subset\bn,\ \abs{A}=n\Big\}\ .
  \]
  It is clear that $(e_i)$ is a $1$-unconditional basis in
  $\tnorm{\cdot}$. We next prove that it also
  satisfies Property~(A) with constant $1+\vare$. Fix
  $x\in\coo$ and $B,\Bt\subset\bn\setminus \supp(x)$ such that
  $\norm{x}_{\ell_\infty}\leq 1$ and $\abs{B}=\abs{\Bt}<\infty$. It
  will be sufficient to prove that $\tnorm{x+\bi_B}\leq (1+\vare)
  \tnorm{x+\bi_{\Bt}}$. We may of course assume that $x\geq 0$.

  Let $n\in\bn$, $A\subset\bn$ and $x^*\in \vare B_{X^*}$ be such that
  $\abs{A}=n$ and
  \begin{equation}
    \label{eq:norming-x+B}
  \begin{aligned}
    \tnorm{x+\bi_B}&=
    \bigip{x+\bi_B}{x^*+{\ts\frac{1}{\varf^*(n)}}\bi_A}\\[2ex]
    &=\ip{x}{x^*}+\ip{\bi_B}{x^*}+
    {\ts\frac{1}{\varf^*(n)}} \ip{x}{\bi_A} +
    {\ts\frac{1}{\varf^*(n)}} \abs{B\cap A}\ .
  \end{aligned}
  \end{equation}
  Without loss of generality we may assume that $\supp(x^*)\cup
  A\subset \supp(x)\cup B$, and hence $x^*\geq 0$. Note that
  \begin{equation}
    \label{eq:x*-on-B}
    \begin{aligned}
      \ip{\bi_B}{x^*} &\leq \vare\norm{\bi_B} =
      \vare\varf(\abs{B})=\vare\varf(\abs{\Bt}) \\[2ex]
      &=
      \vare
      \bigip{x+\bi_{\Bt}}{{\ts\frac{1}{\varf^*\big(\abs{\Bt}\big)}}\bi_{\Bt}}
      \leq \vare\tnorm{x+\bi_{\Bt}}\ .
    \end{aligned}
  \end{equation}
  Now choose $\At\subset\bn$ such that $\At\cap\supp(x)=A\cap\supp(x)$
  and $\abs{\Bt\cap \At}=\abs{B\cap A}$. Then $\abs{\At}=n$ and
  \begin{equation}
    \label{eq:action-of-At}
    \ip{x}{\bi_A} + \abs{B\cap A}=\ip{x}{\bi_{\At}} + \abs{\Bt\cap \At}\
    .
  \end{equation}
  We now obtain
  \begin{align*}
    \tnorm{x+\bi_B&} = \ip{x}{x^*} + \ip{\bi_B}{x^*} +
    {\ts\frac{1}{\varf^*(n)}} \ip{x}{\bi_A} +
    {\ts\frac{1}{\varf^*(n)}} \abs{B\cap A} &
    \text{by~\eqref{eq:norming-x+B}}\\[2ex]
    &\leq \ip{x}{x^*} + \vare\tnorm{x+\bi_{\Bt}} +
    {\ts\frac{1}{\varf^*(n)}} \ip{x}{\bi_{\At}} +
    {\ts\frac{1}{\varf^*(n)}} \abs{\Bt\cap \At} &
    \text{by~\eqref{eq:x*-on-B} and~\eqref{eq:action-of-At}} \\[2ex]
    &\leq 
    \bigip{x+\bi_{\Bt}}{x^*+{\ts\frac{1}{\varf^*(n)}}\bi_{\At}} +
    \vare\tnorm{x+\bi_{\Bt}} & \text{as }x^*\geq 0\\[2ex]
    &\leq (1+\vare) \tnorm{x+\bi_{\Bt}}\ ,
  \end{align*}
  as required. Thus, so far, we have that $(e_i)$ is $1$-unconditional
  and $(1+\vare)$-greedy in $\tnorm{\cdot}$. It is also clear that
  $\tnorm{e_i}=1+\vare$ for all $i\in\bn$. Let $\psi$ and
  $\psi^*$ denote the fundamental and, respectively, dual fundamental
  function of $(e_i)$ with respect to $\tnorm{\cdot}$. Let
  $m,n\in\bn$ and $A,B\subset\bn$ with $\abs{A}=m$ and $\abs{B}=n$. By
  definition of $\tnorm{\cdot}$, in the dual space we have
  $\tnorm{\frac{1}{\varf^*(m)}\bi_A}^*\leq 1$, from which it follows that
  $\psi^*(m)\leq \varf^*(m)$. Also, for any $x^*\in\vare B_{X^*}$ we
  have
  \begin{align*}
    \Bigip{\bi_B}{x^*+{\ts\frac{1}{\varf^*(m)}} \bi_A} & \leq
    \vare\norm{\bi_B} + {\ts\frac{\varf(m)}{m}} \abs{B\cap A}\\[2ex]
    &\leq 
    \vare\varf(n) + {\ts\frac{\varf(\abs{B\cap A})}{\abs{B\cap A}}}
    \abs{B\cap A} \leq (1+\vare) \varf(n) .
  \end{align*}
  Hence $\psi(n)\leq(1+\vare)\varf(n)$, and $(e_i)$ is
  $(1+\vare)$-bidemocratic in $\tnorm{\cdot}$. So if we replace
  $\tnorm{\cdot}$ with $\frac{1}{1+\vare}\tnorm{\cdot}$ and apply
  Theorem~\ref{thm:bidemocratic-renorming}, then we obtain a new norm
  $(1+\vare)$-equivalent to $\tnorm{\cdot}$, and 
  with respect to which $(e_i)$ is normalized, $1$-unconditional,
  $1$-bidemocratic and $(1+\vare)^2$-greedy.
\end{proof}
Let us now observe that our theorem applies to a large class of Banach
spaces and bases. We say that a democratic basis $(e_i)$ (or its
fundamental function $\varf$) has the \emph{upper regularity property
}(or URP for short) if there exists an integer $r>2$ such that
\[
\varf(rn)\leq {\ts \frac12} r\varf(n)\qquad \text{for all }n\in\bn\
.
\]
This is easily seen to be equivalent to the existence of $0<\beta<1$
and a constant $C$ such that
\[
\varf(n)\leq C \big( {\ts \frac{n}{m} }\big)^\beta \varf(m)\qquad
\text{for all }m\leq n\ .
\]
This property was introduced in~\cite{dkkt:03} where it was shown that
a greedy basis of a Banach space with nontrivial type has the URP and
that a greedy basis with the URP is bidemocratic. More precisely, they
showed that if $(e_i)$ is a greedy basis with fundamental function
$\varf$, and there exists a constant $C$ such that
\begin{equation}
  \label{eq:weak-urp}
  \sum_{k=1}^n \frac{\varf(n)}{\varf(k)} \leq Cn\qquad\text{for all
  }n\in\bn\ ,
\end{equation}
then $(e_i)$ is bidemocratic. It is of course clear that the URP
implies~\eqref{eq:weak-urp}.

It is well known that $L_p[0,1]$ for $1<p<\infty$ has nontrivial
type. Thus we obtain the following corollary.
\begin{cor}
  Let $1<p<\infty$. For all $\vare>0$ there is an equivalent norm on
  $L_p[0,1]$ in which the Haar basis is normalized,
  $1$-unconditional, $1$-bidemocratic and $(1+\vare)$-greedy.
\end{cor}
\begin{rem}
  By the Albiac-Wojtaszczyk characterization, a $1$-greedy basis is
  suppression $1$-unconditional, and hence $2$-unconditional. As shown
  in~\cite{dosz:11}*{Theorem~4.1}, the unconditional constant~$2$
  is in general the best one can say about a $1$-greedy basis. This is why
  the $1$-unconditionality was included in the above results.
\end{rem}

\section{The class of quasi-concave functions}
\label{sec:fundamental-functions}

We denote by $\br^+$ the set of (strictly) positive real
numbers. Recall that the fundamental function $\varf\colon\bn\to\br^+$
of a basis of a Banach space is increasing and $n\mapsto
\frac{\varf(n)}{n}$ is decreasing. Let us now call a function
$\varf\colon [1,\infty)\to\br^+$ defined on the \emph{real }interval
$[1,\infty)$ a \emph{fundamental function }if it is increasing and
$x\mapsto \frac{\varf(x)}{x}$ is decreasing. Observe that every
fundamental function $\varf$ is subadditive. Indeed, for
$x,y\in[1,\infty)$ we have
\[
\varf(x+y) = \frac{\varf(x+y)}{x+y} \cdot x +\frac{\varf(x+y)}{x+y}
\cdot y \leq \frac{\varf(x)}{x} \cdot x +\frac{\varf(y)}{y}
\cdot y =\varf(x)+\varf(y)\ .
\]
The fundamental function of a basis of a Banach space \emph{is }the
restriction to $\bn$ of a fundamental function in the above
sense. Indeed, if $\varf\colon\bn\to\br^+$ is the fundamental function
of a basis, then we can extend it to a function on $[1,\infty)$ by
linear interpolation. A straightforward calculation shows that this
extended function is a fundamental function in the above sense. The
converse is also true, \ie if $\varf\colon [1,\infty)\to\br^+$ is a
fundamental function, then its restriction to $\bn$ is the fundamental
function of a basis. This will be shown in
Proposition~\ref{prop:space-for-fund-fn} at the end of this section.
Given a fundamental function $\varf\colon [1,\infty)\to\br^+$ and a
basis $(e_i)$ of a Banach space, we say that $\varf$ is \emph{a
  fundamental function for }$(e_i)$ if the restriction of $\varf$ to
$\bn$ is the fundamental function of $(e_i)$.
\begin{rem}
  In the literature fundamental functions in the above sense are also known as
  \textit{quasi-concave functions}. See for
  example~\cite{bennett-sharpley:88}*{Definition~5.6 on page 69},
  where quasi-concave functions are defined on the interval
  $[0,\infty)$ and are naturally associated with
  rearrangement-invariant spaces. Since we work with discrete lattices
  corresponding to unconditional bases which in general are not
  symmetric, for us it will be more convenient to work with the
  definition above instead.
\end{rem}
We will now introduce a parameter $\delta$ which provides information
on the growth of fundamental functions. After that we will show that
the concave envelope of a fundamental function is also a fundamental
function.

Let $\varf\colon [1,\infty)\to\br^+$ be a fundamental function. It
will sometimes be more convenient to work with the function
$\lambda\colon[1,\infty)\to\br^+$ defined by
$\lambda(x)=\frac{\varf(x)}{x}$. Note that $\lambda$ is decreasing and
$x\lambda(x)$ is increasing. For $y\in [1,\infty)$ define
\[
\delta_\varf(y)= \liminf _{x\to\infty} \frac{\varf(yx)}{y\varf(x)} =
\liminf _{x\to\infty} \frac{\lambda(yx)}{\lambda(x)}\ .
\]
It follows from properties of $\lambda$ that $\delta_\varf$ is
decreasing and bounded above by~$1$. Hence
\[
\delta(\varf)=\inf_{y\geq 1}\delta_\varf(y)=\lim_{y\to\infty}
\delta_\varf(y)\in [0,1]\ .
\]
Let us now observe that the function $\delta_\varf$ and the parameter
$\delta(\varf)$ depend only on  the values of $\varf$ on $\bn$. Fix
$m\in\bn$. For any real $x\in[1,\infty)$, putting $n=\intp{x}+1$, we
have
\[
\frac{\lambda(mx)}{\lambda(x)} = x\cdot\frac{\lambda(mx)}{x\lambda(x)}
\geq x\cdot \frac{\lambda(mn)}{n\lambda(n)} = \frac{x}{\intp{x}+1}
\cdot \frac{\lambda(mn)}{\lambda(n)}\ .
\]
It follows that for each $y\in [1,\infty)$ we have
\[
\inf _{n\in\bn,\ n\geq \intp{y}+1} \frac{\lambda(mn)}{\lambda(n)}
\geq \inf _{x\in\br,\ x\geq y} \frac{\lambda(mx)}{\lambda(x)} \geq
\frac{\intp{y}}{\intp{y}+1} \cdot \inf _{n\in\bn,\ n\geq \intp{y}+1}
\frac{\lambda(mn)}{\lambda(n)}\ ,
\]
and hence we obtain
\[
\delta_\varf(m)=\liminf _{n\to\infty} \frac{\lambda(mn)}{\lambda(n)}\ .
\]
Thus we have 
\[
\delta(\varf)=\lim_{m\to\infty} \liminf _{n\to\infty}
\frac{\lambda(mn)}{\lambda(n)}\ .
\]
One consequence of all this is that if $(e_i)$ is a basis of a
Banach space, then the parameter $\delta(\varf)$ is the same for
\emph{any }fundamental function $\varf$ for $(e_i)$.

Two fundamental functions $\varf$ and $\psi$ are said to be
\emph{equivalent }if there exist positive real numbers $a$ and $b$
such that $a\varf(x)\leq\psi(x)\leq b\varf(x)$ for all $x\in
[1,\infty)$. In this case we write $\varf\sim\psi$. Note that
equivalence also only depends on the restrictions to $\bn$ of $\varf$
and $\psi$. Indeed, if for some $b>0$ we have $\psi(n)\leq b\varf(n)$
for all $n\in\bn$, then
\[
\psi(x)=x\cdot \frac{\psi(x)}{x} \leq x\cdot
\frac{\psi(\intp{x})}{\intp{x}} \leq 2b\varf(\intp{x})\leq 2b\varf(x)
\]
for all $x\in[1,\infty)$. So for a basis $(e_i)$ of a Banach
space with a fundamental function $\varf$, the property of having
$\delta(\varf)>0$ is invariant under renormings. We now prove a result
about fundamental functions with positive $\delta$-parameter. This
will be used in Theorem~\ref{thm:non-flat-fundamental-fn} in the next
section.
\begin{lem}
  \label{lem:non-flat-fund-fn}
  Let $\varf$ be a fundamental function with $\delta(\varf)>0$. Then
  for all $\vare>0$ and for all $m\in\bn$ there exists a fundamental
  function $\psi\sim\varf$ such that $\delta_\psi(m)>
  \frac{1}{1+\vare}$.
\end{lem}
\begin{proof}
  It is enough to show that if $\delta_\varf(m^2)>\delta$, then there
  exists a fundamental function
  $\psi\sim\varf$ such that $\delta_\psi(m)>\sqrt{\delta}$. Indeed,
  assuming this result, we fix $0<\delta<\delta(\varf)$, choose
  $k\in\bn$ with $\delta^{\frac{1}{2^k}}>\frac{1}{1+\vare}$, and obtain
  fundamental functions $\varf=\varf_0\sim\varf_1\sim\dots\sim\varf_k$
  such that $\delta_{\varf_j}\big( m^{2^{k-j}}
  \big)>\delta^{\frac{1}{2^j}}$ for $j=0,1,2,\dots,k$. Putting
  $\psi=\varf_k$ completes the proof.

  Set $\lambda(x)=\frac{\varf(x)}{x}$, $x\in[1,\infty)$. To prove our
  initial claim, choose $n_0\in\bn$ such that $\lambda(m^2x)>\delta
  \lambda(x)$ for all real $x\geq n_0$. We now define a new
  function $\mu\colon[1,\infty)\to\br^+$ as follows. We set
  $\mu(x)=\lambda(x)$ for all real $x\in[1,n_0]$ and for all integers
  $x$ of the form $x=m^{2k}n_0$, $k=0,1,2,\dots$. We then extend the
  definition of $\mu$ by interpolation as follows. Given a real number
  $x\in [n_0,\infty)$, we fix an integer $k\geq 0$ such that $x\in
  [n,m^2n]$, where $n=m^{2k}n_0$. (Note that $k$ is unique unless
  $x\in \{m^{2j}n_0:\, j\in\bn\}$.) Then there is a unique
  $\theta\in[0,1]$ such that $x=n^{1-\theta}(m^2n)^\theta$. We define
  \[
  \mu (x)= \lambda(n)^{1-\theta} \lambda(m^2n)^{\theta}\ .
  \]
  Note that for $x=n$ and $x=m^2n$ this agrees with the previous
  definition of $\mu(x)=\lambda(x)$. It follows that $\mu(x)$ is
  well-defined, and in particular it does not depend on the choice of $k$
  when $x\in \{m^{2j}n_0:\, j\in\bn\}$. We now prove the following
  properties for each integer $k\geq 0$ with $n=m^{2k}n_0$.
  \begin{mylist}{(iii)}
  \item
    $\mu(x)$ is decreasing and $x\mu(x)$ is increasing on $[n,m^2n]$.
  \item
    $\delta \lambda(x)\leq\mu(x)\leq m^2\lambda(x)$ for all
    $x\in [n,m^2n]$,
  \item
    $\mu (mx)\geq \sqrt{\delta}\mu(x)$ for all $x\in [n,m^2n]$.
  \end{mylist}
  We will then set $\psi(x)=x\mu(x)$ for each $x\in[1,\infty)$. Since
  $\mu=\lambda$, and hence $\psi=\varf$, on the set
  $[1,n_0]\cup \{m^{2k}n_0:\,k\geq 0\}$, property~(i) implies
  that $\psi$ is a fundamental function, which is equivalent to
  $\varf$ by~(ii), and satisfies $\delta_\psi(m)\geq\sqrt{\delta}$
  by~(iii). This proves the initial claim, and hence the lemma.

  To see (i) simply differentiate the functions
  \[
  \lambda(n)^{1-\theta} \lambda(m^2n)^{\theta}\qquad
  \text{and}\qquad n^{1-\theta}
  (m^2n)^{\theta}\lambda(n)^{1-\theta}\lambda(m^2n)^{\theta}
  \]
  with respect to~$\theta$.

  Next, fix $\theta\in[0,1]$ and set $x=n^{1-\theta}(m^2n)^\theta$. By
  the properties of $\lambda$, we have
  \begin{align*}
    \mu(x) &= \lambda(n)^{1-\theta} \lambda(m^2n)^{\theta} \leq
    \lambda (n) = n\lambda(n)\cdot\frac{1}{n} \leq
    x\lambda(x)\cdot\frac{1}{n}\\[2ex]
    &= n^{1-\theta} (m^2n)^{\theta} \cdot \lambda (x) \cdot\frac{1}{n}
    = m^{2\theta} \lambda(x) \leq m^2 \lambda (x) \ ,
  \end{align*}
  and, since $\lambda(m^2n)\geq \delta \lambda (n)$, we have
  \[
  \mu(x)=\lambda(n)^{1-\theta} \lambda(m^2n)^{\theta} \geq \delta
  ^{\theta} \lambda(n) \geq \delta \lambda(x)\ .
  \]
  Hence (ii) follows. Finally, fix $0\leq\theta\leq 1$. In
  order to verify~(iii) we need to show that
  \begin{equation}
    \label{eq:sqrt-delta}
    \frac{\mu \big( m\cdot n^{1-\theta} (m^2n)^{\theta}
      \big)}{\mu \big( n^{1-\theta} (m^2n)^{\theta} \big)}
    \geq\sqrt{\delta}\ .
  \end{equation}
  We consider two cases. When $0\leq\theta \leq \frac{1}{2}$, we can write
  $m\cdot n^{1-\theta} (m^2n)^{\theta} =  n^{1-\theta'}
  (m^2n)^{\theta'}$, where $\theta+\frac{1}{2}=\theta'$. Then, since
  $\lambda(m^2n)\geq \delta\lambda(n)$ and $\theta'-\theta>0$, the
  left-hand side of~\eqref{eq:sqrt-delta} becomes
  \[
  \frac{\lambda(n)^{1-\theta'}
    \lambda(m^2n)^{\theta'}}{\lambda(n)^{1-\theta}
    \lambda(m^2n)^{\theta}}  \geq \lambda(n)^{\theta-\theta'}
  \cdot \big( \delta\lambda(n)\big)^{\theta'-\theta} =
  \sqrt{\delta}\ .
  \]
  In the second case, when $\frac{1}{2}\leq \theta\leq 1$, we have $m\cdot
  n^{1-\theta} (m^2n)^{\theta} =  (m^2n)^{1-\theta'}
  (m^4n)^{\theta'}$, where
  $\theta+\frac{1}{2}=1+\theta'$. Then the left-hand side
  of~\eqref{eq:sqrt-delta} becomes
  \begin{align*}
    \frac{\lambda(m^2n)^{1-\theta'}
    \lambda(m^4n)^{\theta'}}{\lambda(n)^{1-\theta}
    \lambda(m^2n)^{\theta}}  & \geq
    \frac{\lambda(m^2n)^{1-\theta'}
    \big(\delta\lambda(m^2n)\big)^{\theta'}}{\lambda(n)^{1-\theta}
    \lambda(m^2n)^{\theta}} \\[2ex]
  &= \delta ^{\theta'} \cdot
  \frac{\lambda(m^2n)^{1-\theta}}{\lambda(n)^{1-\theta}} \geq
  \delta^{\theta'+1-\theta} = \sqrt{\delta}\ ,
  \end{align*}
  as required.
\end{proof}
We next prove that every fundamental function is equivalent to a
concave one. This is standard (see for
example~\cite{bennett-sharpley:88}*{Proposition~5.10}), but we repeat
the simple proof here as we need a further property concerning the
$\delta$ parameter.
\begin{lem}
  \label{lem:concave-fund-fn}
  Let $\varf$ be a fundamental function. Then there exists a concave
  fundamental function $\psi$ such that $\varf(x)\leq\psi(x)\leq
  2\varf(x)$ for all $x\in [1,\infty)$. Moreover, we have
  $\delta_\psi(y)\geq \delta_\varf(y)$ for all $y\in [1,\infty)$.
\end{lem}
\begin{proof}
  We let $\psi\colon [1,\infty)\to\br^+$ be the concave envelope of
  $\varf$. Recall that this is the (pointwise) smallest concave
  function dominating $\varf$, and is given by
  \[
  \psi(x)=\sup \sum_{i=1}^n t_i\varf(x_i)\ ,
  \]
  where the supremum is taken over all convex combinations
  $x=\sum_{i=1}^n t_ix_i$ of numbers $x_1,\dots,x_n\in[1,\infty)$. Of
  course this concave envelope exists if and only if the above
  supremum is finite for every $x$. To verify that this is true in our
  case, note that either $x_i<x$ and $\varf(x_i)\leq\varf(x)$, or
  $x_i\geq x$ and we have $\varf(x_i)=\frac{\varf(x_i)}{x_i} x_i\leq
  \frac{\varf(x)}{x} x_i$. It follows that
  \[
  \sum _{i=1}^n t_i\varf(x_i)\leq \varf(x)\sum _{x_i<x} t_i +
  {\ts \frac{\varf(x)}{x} }
  \sum_{x_i\geq x} t_ix_i \leq 2\varf(x)\ .
  \]
  It follows that $\psi$ exists, and $\varf(x)\leq\psi(x)\leq
  2\varf(x)$ for all $x$. We next show that $\psi$ is a fundamental
  function. Let $1\leq x\leq y$, and let $x=\sum_{i=1}^n t_ix_i$ be a
  convex combination of elements of $[1,\infty)$. Then $\sum_{i=1}^n
  t_i(x_i+y-x)=y$, and hence
  \[
  \psi(y)\geq \sum_{i=1}^n t_i\varf(x_i+y-x)\geq \sum_{i=1}^n
  t_i\varf(x_i)\ .
  \]
  Taking supremum yields $\psi(y)\geq \psi(x)$, and so $\psi$ is
  increasing. Next, consider a convex combination $y=\sum_{i=1}^n
  t_iy_i$. Let $z=\frac{x}{y}$. Then $x=\sum_{i=1}^n t_i zy_i$, and
  so
  \begin{align*}
    \frac{\psi(x)}{x} &\geq \frac{1}{x} \sum_{i=1}^n t_i \varf(zy_i) =
    \sum_{i=1}^n t_i \frac{\varf(zy_i)}{zy_i} \cdot
    \frac{y_i}{y}\\[2ex]
    &\geq \sum_{i=1}^n t_i \frac{\varf(y_i)}{y_i} \cdot
    \frac{y_i}{y} =\frac{1}{y} \sum_{i=1}^n t_i\varf(y_i)\ .
  \end{align*}
  After taking supremum, this implies $\frac{\psi(x)}{x} \geq
  \frac{\psi(y)}{y}$, which completes the proof that $\psi$ is a
  fundamental function.

  To show the ``moreover'' part of the lemma, first observe that if
  $\varf$ is bounded, then $\lim _{x\to\infty}\varf(x)$ exists, and hence
  $\delta_\varf(y)=\frac{1}{y}$ for all $y\in[1,\infty)$. Since in
  this case $\psi$ is also bounded, we have
  $\delta_\varf(y)=\delta_\psi(y)=\frac{1}{y}$ for all $y$. Assume now
  that $\varf$ is unbounded. Fix $y\geq 1$. It will be enough to show
  that if $0<\delta<\delta_\varf(y)$ and $\vare>0$, then
  $\delta\leq(1+\vare)\delta_\psi(y)$. Choose $w\in[1,\infty)$ such that
  $\varf(yx)\geq \delta y\varf(x)$ for all $x\geq w$. Since
  $\varf(x)$ tends to infinity, we can then choose $z>w$ such that
  $\frac{\delta y}{\vare}\varf(w)<\varf(z)$.

  We will show that if $x\geq z$, then $(1+\vare) \psi(yx)\geq
  \delta y\psi(x)$, which will complete the proof. Fix a convex
  combination $x=\sum _{i=1}^n t_ix_i$. By the definition of $\psi$,
  we have
  \[
  \psi(yx)\geq \sum_{i=1}^n t_i\varf(yx_i)\ .
  \]
  Let $I=\big\{ i\in\{1,\dots,n\}:\,x_i\geq w\big\}$. By the choice of
  $z$, we have
  \[
  \delta y \sum _{i\notin I} t_i \varf(x_i) \leq \delta y \varf(w)
  < \vare\varf(z)\leq \vare\psi(yx)\ .
  \]
  It follows from this and from the choice of $w$ that
  \[
  \psi(yx)\geq \sum_{i\in I} t_i\varf(yx_i)\geq \delta y
  \sum_{i\in I} t_i \varf(x_i) \geq \delta y \sum_{i=1}^n t_i
  \varf(x_i) - \vare\psi(yx)\ .
  \]
  Since this holds for all convex combinations $x=\sum _{i=1}^k
  t_ix_i$, we obtain $(1+\vare)\psi(yx)\geq \delta y \psi(x)$, as
  required.
\end{proof}
Our next result shows that every fundamental function arises from a
basis in a Banach space.
\begin{prop}
  \label{prop:space-for-fund-fn}
  Let $\varf\colon[1,\infty)\to\br^+$ be a fundamental function. Then
  there is a Banach space with a $1$-unconditional basis $(e_i)$ whose
  fundamental function is the restriction of $\varf$ to $\bn$.
\end{prop}
\begin{proof}
  Let $\cF$ be a family of finite subsets of $\bn$. The only condition
  we impose on $\cF$ that it should contain for every $n\in\bn$ a set
  of size $n$. By scaling we may assume that $\varf(1)=1$. Define a
  norm $\norm{\cdot}$ on the space $\coo$ of finite sequences as
  follows:
  \[
  \norm{x}=\norm{x}_{\ell_\infty} \vee \sup \Big\{ {\ts
    \frac{\varf(\abs{A})}{\abs{A}} } \sum_{i\in A} \abs{x_i}:\,
  A\in\cF\Big\}\ ,\qquad x=(x_i)\in\coo\ .
  \]
 It is clear that $(e_i)$ is a normalized, $1$-unconditional basis of
 the completion $X$ of $\big(\coo,\norm{\cdot}\big)$. Now let
 $m,n\in\bn$, let $A\in\cF$ with $m=\abs{A}$, and let $B\subset\bn$
 with $\abs{B}=n$. Then
 \[
 \frac{\varf(\abs{A})}{\abs{A}} \abs{B\cap A} \leq
 \frac{\varf(\abs{B\cap A})}{\abs{B\cap A}} \abs{B\cap A} =
 \varf(\abs{B\cap A}) \leq \varf(n)\ .
 \]
 It follows that $\bignorm{\sum_{i\in B} e_i}\leq\varf(n)$. On the
 other hand, since $A\in\cF$, we have $\bignorm{\sum_{i\in A}
   e_i}\geq\varf(m)$. Thus the fundamental function of $(e_i)$ is
 indeed $\varf$.
\end{proof}
\begin{rem}
  For a continuous version
  see~\cite{bennett-sharpley:88}*{Proposition~5.8}, where they show
  that every quasi-concave function is the fundamental function of a
  rearrangement-invariant space. However, in our result the basis
  constructed clearly need not be symmetric, or indeed even
  democratic. If for some $\delta>0$ every finite $E\subset\bn$ has a
  subset $A\in\cF$ with $\abs{A}\geq\delta\abs{E}$, then $(e_i)$ is
  $\frac{1}{\delta}$-democratic. If $\cF$ is the set of \emph{all
  }subsets of $\bn$, then $(e_i)$ is bidemocratic.

  It is also possible for $(e_i)$ to be democratic but not
  bidemocratic. For this to happen $\varf$ cannot be arbitrary. For
  example, if $\varf$ has the URP and $(e_i)$ is democratic, then
  $(e_i)$ is automatically
  bidemocratic~\cite{dkkt:03}*{Proposition~4.4}. However, if
  $\varf(n)=n$, say, and $\cF$ is
  the family $\cS$ of Schreier sets, \ie sets $A\subset \bn$ with
  $\abs{A}\leq \min A$, then the dual fundamental function cannot be
  bounded otherwise $X^*$ would be isomorphic to $\co$, and hence $X$
  would be isomorphic to $\ell_1$.
\end{rem}

We will later prove a renorming result for bases with fundamental
function $\varf$ satisfying $\delta(\varf)>0$. We conclude this
section by observing that there are
bases with $\delta(\varf)=0$. Fix integers $1=n_1<n_2<\dots$. Define
$\varf\colon\bn\to\br^+$ by setting $\varf(1)=1$ and keeping
$\frac{\varf(n)}{n}$ constant on intervals $[n_k,n_{k+1}]$ when $k$ is
odd, and keeping $\varf$ constant on intervals $[n_k,n_{k+1}]$ when
$k$ is even. Extend $\varf$ to a fundamental function defined on
$[1,\infty)$. If the $n_k$ are sufficiently rapidly increasing, then
$\delta_\varf(m)=0$ for all $m\in\bn$. By
Proposition~\ref{prop:space-for-fund-fn} this $\varf$ is a
fundamental function for some Schauder basis.

\section{The general case}
\label{sec:general}

In this section we will prove Theorem~\ref{mainthm:1+e-democratic} and
give a positive answer to
Problem~\ref{problem:1+e-greedy-renorming-general} in the case the
fundamental function $\varf$ has $\delta(\varf)>0$. We will require
the following crucial lemma.
\begin{lem}
  \label{lem:flat-norming-fnl}
  Let $(e_i)$ be a normalized, $1$-unconditional, $\Delta$-democratic
  basis of a Banach space $X$ with fundamental function $\varf$. Given
  $0<q<1$, fix $C>\frac{\Delta}{q(1-q)}$ and set
  \[
  \cA= \Big\{ A\subset\bn:\, A\text{ finite and }\bignorm{{\ts
      \frac{\varf(\abs{A})}{\abs{A}}} \bi_A}^*\leq C \Big\}\ .
  \]
  Then for every finite $E\subset\bn$ there exists $A\in\cA$ such that
  $A\subset E$ and $\abs{A}\geq q\abs{E}$.
\end{lem}

\begin{proof}
  Choose $\delta>0$ such that $C>\frac{(1+\delta)\Delta}{\delta
    q(1-q)}$. Let $n\in\bn$ and $E\subset \bn$ with $\abs{E}=n$. We
  inductively construct $z^{(1)}, z^{(2)}, \dots$ in $B_{X^*}$ and
  pairwise disjoint subsets $E_1, E_2,\dots$ of $E$ as follows. Assume
  that for some $k\in\bn$ we have already
  defined $z^{(1)},\dots, z^{(k-1)}$ and $E_1,\dots, E_{k-1}$. Set
  $F_k=E\setminus \cup_{i=1}^{k-1}E_i$ (so, in particular,
  $F_1=E$). If $\abs{F_k}<(1-q)n$, then we
  stop. Otherwise we choose $z^{(k)}\geq 0$ in $B_{X^*}$ satisfying
  \[
  \supp z^{(k)}\subset F_k \qquad\text{and}\qquad
  \ip{\bi_{F_k}}{z^{(k)}}=\norm{\bi_{F_k}}\ ,
  \]
  and define
  \[
  E_k=\{ i\in F_k:\,z^{(1)}_i+\dots+z^{(k)}_i\geq\delta \}\ .
  \]
  This completes the induction step. We will see in a moment that this
  process terminates after a finite number of steps.  Assume that
  $E_1,\dots, E_m$ and $F_1,\dots, F_{m+1}$ have been defined for some
  $m\geq 1$ (note that $\abs{F_1}=\abs{E}>(1-q)n$, so at least one set
  $E_1$ is defined). For each $k=1,\dots, m$ and for each $i\in E_k$
  we have
  \[
  z^{(1)}_i+\dots+z^{(k-1)}_i<\delta\ ,
  \]
  and hence
  \[
  z^{(1)}_i+\dots+z^{(m)}_i= z^{(1)}_i+\dots+z^{(k)}_i<1+\delta\ .
  \]
  We also have
  \[
  z^{(1)}_i+\dots+z^{(m)}_i <\delta \qquad\text{for all }i\in F_{m+1}\ .
  \]
  It follows that
  \begin{align*}
    \bigip{\bi_E}{z^{(1)}+\dots+z^{(m)}} & =\sum _{k=1}^m
    \bigip{\bi_{E_k}}{z^{(1)}+\dots +z^{(m)}}\\[2ex]
    & \phantom{=} +\bigip{\bi_{F_{m+1}}}{z^{(1)}+\dots +z^{(m)}} <
    (1+\delta)n\ .
  \end{align*}
  On the other hand, since $\abs{F_k}\geq (1-q)n$ for each
  $k=1,\dots,m$, and since $\varf(x)/x$ is decreasing, we have
  \[
  \bigip{\bi_E}{z^{(1)}+\dots+z^{(m)}} =\sum _{k=1}^m
  \bigip{\bi_{F_k}}{z^{(k)}} \geq
  m\frac{\varf\big((1-q)n\big)}{\Delta}\geq m(1-q)\frac{\varf(n)}{\Delta}
  \ .
  \]
  Thus, we can deduce that
  \begin{equation}
    \label{eq:flat-norming-fnl:bound-on-m}
    m\leq \frac{(1+\delta)\Delta}{(1-q)}\cdot\frac{n}{\varf(n)}\ ,
  \end{equation}
  which in particular shows that the process does indeed
  terminate. Let $m$ denote the time when this happens, \ie when
  $\abs{F_{m+1}}<(1-q)n$. Let us now set $A=\bigcup_{k=1}^m
  E_k$. It is clear that $\abs{A}\geq qn$. It remains to show that
  $A\in \cA$. Since
  \[
  z^{(1)}_i+\dots +z^{(m)}_i \geq\delta\qquad\text{for all }i\in A\ ,
  \]
  it follows that $\norm{\delta \bi_A}^*\leq \norm{z^{(1)}+\dots
    +z^{(m)}}^*\leq m$. Combining this observation
  with~\eqref{eq:flat-norming-fnl:bound-on-m} above, we obtain
  \[
  \bignorm{{\ts \frac{\varf(\abs{A})}{\abs{A}}} \bi_A}^*
  \leq \frac{m\varf(\abs{A})}{\delta\abs{A}} \leq
  \frac{(1+\delta)\Delta}{(1-q)\delta}\cdot
  \frac{n}{\abs{A}} \cdot \frac{\varf(\abs{A})}{\varf(n)}
  \leq \frac{(1+\delta)\Delta}{(1-q)\delta}\cdot
  \frac{1}{q} \cdot \frac{\varf(\abs{A})}{\varf(n)}
  \leq C\ ,
  \]
  which completes the proof.
\end{proof}

We are now ready to prove Theorem~\ref{mainthm:1+e-democratic} on
improving the democracy constant.
\begin{thm}
  \label{thm:1+e-democratic}
  Let $(e_i)$ be an unconditional and democratic basis of a Banach
  space $X$. For any $\vare>0$ there is an equivalent norm
  $\tnorm{\cdot}$ on $X$ with
  respect to which $(e_i)$ is normalized, $1$-unconditional and
  $(1+\vare)$-democratic.
\end{thm}
\begin{proof}
  We can assume that $(e_i)$ is a normalized, $1$-unconditional basis. Let
  $\Delta$ be the democracy constant and $\varf$ be a fundamental
  function for $(e_i)$. Given $\vare>0$, set $q=\frac{1}{1+\vare}$, fix
  $C>\frac{\Delta}{q(1-q)}$, and let $\cA$ be the family given by
  Lemma~\ref{lem:flat-norming-fnl}. Then the following defines a
  $C$-equivalent norm on $X$:
  \[
  \tnorm{x} = \norm{x}\vee \sup\Big\{ \bigip{\abs{x}}{\
    {\ts \frac{\varf(\abs{A})}{\abs{A}}} \bi_{A}}:\,A\in\cA \Big\}\ .
  \]
  Clearly, $(e_i)$ is still normalized and $1$-unconditional in the
  new norm. We need to verify that it is $(1+\vare)$-democratic. Fix
  $E\subset \bn$ and let $n=\abs{E}$. Taking $A\in\cA$ with $A\subset
  E$ and $\abs{A}\geq q\abs{E}$, we obtain
  \[
  \tnorm{\bi_E}\geq
  \Bigip{\bi_E}{{\ts\frac{\varf(\abs{A})}{\abs{A}}}\bi_{A}} =
  \frac{\varf(\abs{A})}{\abs{A}}\cdot\abs{A} \geq
  \frac{\varf(\abs{E})}{\abs{E}}\cdot\abs{A} \geq q\varf(n)\ .
  \]
  It remains to verify that $\tnorm{\bi_E}\leq\varf(n)$. On the one
  hand, by definition, we have $\norm{\bi_E}\leq
  \varf(n)$. On the other hand, for an arbitrary $A\in \cA$
  we have
  \begin{align*}
    \Bigip{\bi_E}{{\ts\frac{\varf(\abs{A})}{\abs{A}}}\bi_{A}} &=
    \frac{\varf(\abs{A})}{\abs{A}}\cdot\abs{A\cap E} \\[2ex]
    &\leq \frac{\varf(\abs{A\cap E})}{\abs{A\cap E}}\cdot\abs{A\cap E}
    \leq \varf(n)\ .
  \end{align*}
\end{proof}
\begin{rem}
  The upper bound on the equivalence constant on $\tnorm{\cdot}$ given
  by the proof of Theorem~\ref{thm:1+e-democratic} above (which in
  turn comes from the proof of Lemma~\ref{lem:flat-norming-fnl}) is of
  order $\frac{1}{\vare}$. In special cases this can be improved. For
  example, it is not hard to see that in Tsirelson's space we get a
  constant of order $\log \frac{1}{\vare}$. However, in general, the
  best constant must converge to infinity as $\vare$ goes to
  zero. Indeed, assume that $(e_i)$ is a greedy basis of $X$ for which
  there exists a constant $C$ such that for all $\vare>0$ there is a
  $C$-equivalent norm $\norm{\cdot}_\vare$ on $X$ with respect to
  which $(e_i)$ is normalized and $(1+\vare)$-democratic. Fix a
  non-trivial ultrafilter $\cU$ and define $\tnorm{x}=\lim _{\cU}
  \norm{x}_{\frac{1}{n}}$ for $x\in X$. Then $\tnorm{\cdot}$ is a
  $C$-equivalent norm on $X$ with respect to which $(e_i)$ is
  $1$-democratic. As mentioned in the Introduction, there are greedy
  bases for which such renorming is not possible.
\end{rem}

\begin{thm}
  \label{thm:non-flat-fundamental-fn}
  Let $(e_i)$ be a greedy basis of a Banach space $X$ with fundamental
  function $\varf$. Assume that $\delta(\varf)>0$. Then for all
  $\vare>0$ there is an equivalent norm on $X$ with respect to which
  $(e_i)$ is normalized, $1$-unconditional and $(1+\vare)$-greedy.
\end{thm}
\begin{proof}
  We can assume that $(e_i)$ is normalized and $1$-unconditional. Let
  $\Delta$ be the democracy constant of $(e_i)$. Given $\vare>0$, set
  $q=\frac{1}{1+\vare}$, fix $C>\frac{\Delta}{q(1-q)}$, and let $\cA$
  be the family given by Lemma~\ref{lem:flat-norming-fnl}. Next, fix
  $m\geq 2$ in $\bn$ such that $\frac{m}{m-1}\leq 1+\vare$. With the
  given $\vare$ and $m$ we apply Lemma~\ref{lem:non-flat-fund-fn}
  and then Lemma~\ref{lem:concave-fund-fn} to obtain a concave
  fundamental function $\psi$ and positive constants $a$ and $b$ such
  that $\delta_\psi(m)>q$ and $a\varf(x)\leq\psi(x)\leq b\varf(x)$ for
  all $x\in[1,\infty)$. By the definition of $\delta_\psi$, we can
  choose an integer $n_0>\frac{1}{\vare}$ such that
  $\psi(x)>qm\psi\big(\frac{x}{m}\big)$ for all $x\geq n_0$. Set
  $s=\frac{\vare a}{1+\vare}$, $L=\frac{m\psi(1)}{\vare}$ and
  \[
  \cF_m=\Big\{ \sum_{i=1}^m {\ts \frac{\psi(\abs{A_i})}{\abs{A_i}} }
  \bi_{B_i}:\,B_i\subset A_i\in\cA\ ,\ A_1,\dots, A_m
  \text{ pairwise disjoint} \Big\}\ .
  \]
  We are now ready to define a new norm $\tnorm{\cdot}$ as follows.
  \[
  \tnorm{x} = \sup \big\{
  \ip{\abs{x}}{x^*+f+L\bi_A}:\,x^*\in s B_{X^*},\ f\in\cF_m,\
  \abs{A}\leq n_0\big\}\ .
  \]
  It is easy to verify that $s\norm{x}\leq \tnorm{x} \leq
  (s+mbC+Ln_0) \norm{x}$ for all $x\in X$, and it is clear that
  $(e_i)$ is a $1$-unconditional basis in $\tnorm{\cdot}$. We next
  prove that it also satisfies Property~(A) with constant
  $1+4\vare$. Fix $x\in\coo$ with $x\geq 0$ and
  $B,\Bt\subset\bn\setminus \supp(x)$ such that
  $\norm{x}_{\ell_\infty}\leq 1$ and $\abs{B}=\abs{\Bt}<\infty$. It is
  sufficient to prove that $\tnorm{x+\bi_B}\leq (1+4\vare)
  \tnorm{x+\bi_{\Bt}}$.

  For some $x^*\in sB_{X^*}$, $f=\sum_{i=1}^m
  \frac{\psi(\abs{A_i})}{\abs{A_i}} \bi_{B_i}\in\cF_m$ and
  $A\subset\bn$ with $\abs{A}\leq n_0$ we have
  \begin{equation}
    \label{eq:norming-x+B-2}
  \begin{aligned}
    \tnorm{x+\bi_B}&=
    \ip{x+\bi_B}{x^*+f+L\bi_A}\\[2ex]
    &= \ip{x}{x^*}+\ip{\bi_B}{x^*} + \ip{x}{f} + \ip{\bi_B}{f} +
    L\ip{x}{\bi_A} + L\abs{B\cap A}\ .
  \end{aligned}
  \end{equation}  
  Without loss of generality we may assume that $x^*\geq 0$.
  We now estimate some of the terms above. First, we have
  \[
  \ip{\bi_B}{x^*}\leq s\norm{\bi_B}\leq s\varf(\abs{B}) \leq
  \frac{s}{a} \psi(\abs{\Bt})\ .
  \]
  On the other hand, we can choose $C\subset\Bt$ such that $C\in\cA$
  and $\abs{C}\geq q\abs{\Bt}$. Then $g=\frac{\psi(\abs{C})}{\abs{C}}
  \bi_C\in\cF_m$, and so
  \[
  \tnorm{x+\bi_{\Bt}} \geq \ip{x+\bi_{\Bt}}{g} = {\ts
  \frac{\psi(\abs{C})}{\abs{C}} } \abs{C} \geq {\ts
  \frac{\psi(\abs{\Bt})}{\abs{\Bt}} } \abs{C} \geq q\psi(\abs{\Bt})\ .
  \]
  Hence, by the choice of $s$, we have
  \begin{equation}
    \label{eq:x*-on-B-2}
    \ip{\bi_B}{x^*}\leq \frac{s}{a} \psi(\abs{\Bt})\leq
    \frac{s(1+\vare)}{a} \tnorm{x+\bi_{\Bt}} = 
    \vare\tnorm{x+\bi_{\Bt}}\ .
  \end{equation}  
  Next, without loss of generality, we may assume that
  \[
  {\ts \frac{\psi(\abs{A_m})}{\abs{A_m}} } \ip{x}{\bi_{B_m}} = \min
  _{1\leq i\leq m} {\ts \frac{\psi(\abs{A_i})}{\abs{A_i}} }
  \ip{x}{\bi_{B_i}}\ ,
  \]
  and hence we obtain
  \begin{equation}
    \label{eq:f-on-x-2}
    \ip{x}{f} = \sum_{i=1}^m {\ts \frac{\psi(\abs{A_i})}{\abs{A_i}} }
    \ip{x}{\bi_{B_i}} \leq \frac{m}{m-1}
    \sum_{i=1}^{m-1} {\ts \frac{\psi(\abs{A_i})}{\abs{A_i}} }
    \ip{x}{\bi_{B_i}}\ .
  \end{equation}
  Using the fact that $\frac{\psi(x)}{x}$ is decreasing, we then
  obtain
  \begin{equation}
    \label{eq:f-on-B}
    \ip{\bi_B}{f}=\sum_{i=1}^m \frac{\psi(\abs{A_i})}{\abs{A_i}}
    \abs{B_i\cap B} \leq \sum_{i=1}^m \psi(\abs{B_i\cap B})\ .
  \end{equation}
  We now consider two cases. In the first case we assume that
  $\abs{B}=\abs{\Bt}\leq n_0$. Then by the choice of $L$ we have
  \begin{equation}
    \label{eq:f-on-B-1}
    \sum_{i=1}^m \psi(\abs{B_i\cap B}) \leq m\psi(\abs{B}) \leq \vare
    L \abs{\Bt} = \vare \ip{x+\bi_{\Bt}}{L\bi_{\Bt}} \leq
    \vare\tnorm{x+\bi_{\Bt}}\ .
  \end{equation}
  Choose $\At\subset\bn$ such that $\At\cap\supp(x)=A\cap\supp(x)$,
  $\abs{\At\cap\Bt}=\abs{A\cap B}$ and
  $\abs{\At}=\abs{A}$. We then deduce that
  \begin{align*}
    \tnorm{&x+\bi_{\Bt}} &\\[2ex]
    &= \ip{x}{x^*}+\ip{\bi_B}{x^*} + \ip{x}{f} +
    \ip{\bi_B}{f} + L\ip{x}{\bi_A} + L\abs{B\cap A} \quad
    \text{by~\eqref{eq:norming-x+B-2}}\\[2ex]
    &=\ip{x}{x^*}+\ip{\bi_B}{x^*} + \ip{x}{f} +
    \ip{\bi_B}{f} + L\ip{x}{\bi_{\At}} + L\abs{\Bt\cap \At} \quad
    \text{by choice of $\At$}\\[2ex]
    &\leq \ip{x}{x^*}+\ip{x}{f} + L\ip{x}{\bi_{\At}} + L\abs{\Bt\cap
      \At} + 2\vare\tnorm{x+\bi_{\Bt}} \quad
    \text{by~\eqref{eq:x*-on-B-2},~\eqref{eq:f-on-B},~\eqref{eq:f-on-B-1}}
    \\[2ex]
    &\leq \ip{x+\bi_{\Bt}}{x^*+f+L\bi_{\At}} +
    2\vare\tnorm{x+\bi_{\Bt}} \quad \text{as $x^*\geq 0$}\\[2ex]
    &\leq (1+2\vare)\tnorm{x+\bi_{\Bt}} \ .
  \end{align*}
  We now turn to the second case when $\abs{B}=\abs{\Bt}>n_0$. Then by
  concavity of $\psi$ and by the choice of $n_0$ we obtain the
  estimate
  \begin{equation}
    \label{eq:f-on-B-2}
    \sum_{i=1}^m \psi(\abs{B_i\cap B}) \leq m\psi
    \big({\ts \frac{\abs{B}}{m} }\big) \leq (1+\vare)\psi(\abs{B})\ .
  \end{equation}
  Now choose $\At$ as in the previous case, set $\At_i=A_i$ and
  $\Bt_i=B_i\cap\supp(x)$ for $1\leq i<m$, and choose
  $\Bt_m=\At_m\in\cA$ such that $\At_m\subset \Bt$ and
  $\abs{\At_m}\geq q \abs{\Bt}$. Then $g=\sum_{i=1}^m
  \frac{\psi(\abs{\At_i})}{\abs{\At_i}} \bi_{\Bt_i}\in\cF_m$, and
  by~\eqref{eq:f-on-x-2} and the choice of $m$, we have
  \begin{equation}
    \label{eq:g-on-x+Bt}
    \begin{aligned}
      \ip{x+\bi_{\Bt}}{g} & \\[2ex]
      &= \sum_{i=1}^{m-1} {\ts \frac{\psi(\abs{A_i})}{\abs{A_i}} }
      \ip{x}{\bi_{B_i}} + {\ts \frac{\psi(\abs{\At_m})}{\abs{\At_m}} }
      \abs{\At_m}\\[2ex]
      &\geq \frac{m-1}{m} \ip{x}{f} + {\ts
        \frac{\psi(\abs{\Bt})}{\abs{\Bt}} } \abs{\At_m}
      \geq q\ip{x}{f} + q\psi(\abs{\Bt})\ .
    \end{aligned}
  \end{equation}
  It follows that
  \begin{align*}
    \tnorm{&x+\bi_{\Bt}} \\[2ex]
    &= \ip{x}{x^*}+\ip{\bi_B}{x^*} + \ip{x}{f} +
    \ip{\bi_B}{f} + L\ip{x}{\bi_A} + L\abs{B\cap A} \quad
    \text{by~\eqref{eq:norming-x+B-2}}\\[2ex]
    &=\ip{x}{x^*}+\ip{\bi_B}{x^*} + \ip{x}{f} +
    \ip{\bi_B}{f} + L\ip{x}{\bi_{\At}} + L\abs{\Bt\cap \At} \quad
    \text{by choice of $\At$}\\[2ex]
    &\leq \ip{x}{x^*}+\ip{x}{f} + \ip{\bi_B}{f} + L\ip{x}{\bi_{\At}} +
    L\abs{\Bt\cap \At} + \vare\tnorm{x+\bi_{\Bt}} \quad
    \text{by~\eqref{eq:x*-on-B-2}} \\[2ex]
    &\leq \ip{x}{x^*}+ (1+\vare)^2\ip{x+\bi_{\Bt}}{g}\\[2ex]
    & \phantom{\leq\ } +
    L\ip{x}{\bi_{\At}} + L\abs{\Bt\cap \At} + 
    \vare\tnorm{x+\bi_{\Bt}} \quad 
    \text{by~\eqref{eq:f-on-B},~\eqref{eq:f-on-B-2},~\eqref{eq:g-on-x+Bt}}
    \\[2ex]
    &\leq (1+\vare)^2\ip{x+\bi_{\Bt}}{x^*+g+L\bi_{\At}} +\vare
    \tnorm{x+\bi_{\Bt}}\\[2ex]
    &\leq (1+4\vare) \tnorm{x+\bi_{\Bt}}\ ,
  \end{align*}
  as required. Finally, it is easy to see that
  $\tnorm{e_i}=s+\psi(1)+L$ for all $i\in\bn$. So by scaling the new
  norm, we make $(e_i)$ normalized, $1$-unconditional and
  $(1+4\vare)$-greedy.
\end{proof}
The condition $\delta(\varf)>0$ says that the growth of $\varf$ on
intervals of any given fixed size is eventually linear. For example, when
$\varf(x)\sim x$ or $\varf(x)\sim \frac{x}{\log x}$, then
$\delta(\varf)>0$, so Theorem~\ref{thm:non-flat-fundamental-fn}
applies. Note also that when $\varf$ has the URP, then
$\delta(\varf)=0$. However, in that case the basis is bidemocratic and
Theorem~\ref{thm:1+e-greedy-renorming-for-bidemocratic} can be
used. We next give an application of
Theorem~\ref{thm:non-flat-fundamental-fn} in two special cases. Note
that neither of these bases is bidemocratic, so
Theorem~\ref{thm:1+e-greedy-renorming-for-bidemocratic} cannot be
applied.
\begin{cor}
  For all $\vare>0$ there is an equivalent norm on dyadic Hardy space
  $H_1$ and on Tsirelson's space $T$ such that the Haar system,
  respectively, the unit vector basis is normalized,
  $1$-unconditional and $(1+\vare)$-greedy.
\end{cor}

\section{Open problems}

For bidemocratic bases we were able to achieve the best possible
renorming for the democracy constant
(Theorem~\ref{thm:bidemocratic-renorming}). For the greedy constant
Theorem~\ref{thm:1+e-greedy-renorming-for-bidemocratic} gets
arbitrarily close, but the following remains open.
\begin{problem}
  \label{prob:bidem-1-greedy}
  Let $(e_i)$ be a bidemocratic basis of a Banach space $X$. Does
  there exist an equivalent norm on $X$ with respect to which $(e_i)$
  is $1$-greedy?
\end{problem}
The following special case of interest was raised by Albiac and
Wojtaszczyk.
\begin{problem}[\cite{alb-woj:06}*{Problem 6.2}]
  \label{prob:Haar-1-greedy}
  Let $1<p<\infty$. Does there exist an equivalent norm on $L_p[0,1]$
  with respect to which the Haar basis is $1$-greedy?
\end{problem}
The other main problem that remains open concerns the greedy constant
in the general, not necessarily bidemocratic, case.
\begin{problem}
  Let $(e_i)$ be a greedy basis of a Banach space $X$. Does there
  exist for any $\vare>0$ an equivalent norm on $X$ with respect to
  which the basis is $(1+\vare)$-greedy?
\end{problem}
This paper gives a positive answer for a large family of bases. In
terms of the behaviour of the fundamental function $\varf$, if $\varf$
has the URP, or if, on the other extreme, $\delta(\varf)>0$, then the
answer is `yes'. If the basis is bidemocratic, or, more generally, if
there is a constant $C$ such that the family $\cA$ defined in
Lemma~\ref{lem:flat-norming-fnl} consists of \emph{all }finite subsets
of $\bn$, then the proof of
Theorem~\ref{thm:1+e-greedy-renorming-for-bidemocratic} furnishes a
positive answer. However, as pointed out at the end of
Section~\ref{sec:fundamental-functions}, there are fundamental
functions $\varf$ with $\delta(\varf)=0$ and with
\[
\lim_{m\to\infty} \limsup_{n\to\infty} \frac{\varf(mn)}{m\varf(n)}=1\ .
\]
Note that this latter condition rules out properties like the URP. So
there is still a gap.

\end{document}